\documentclass[reqno]{amsart}

\usepackage{amsmath}
\usepackage{amssymb}
\usepackage{amsthm}
\usepackage{xspace}
\usepackage{float}  \restylefloat{figure}
\usepackage{enumerate}
\usepackage{graphicx}
\usepackage{cancel}
\usepackage{pinlabel}
\usepackage{tikz}

\theoremstyle{plain}

\newtheorem*{conjecture}{Conjecture}
\newtheorem{proposition}{Proposition}
\newtheorem{lemma}{Lemma}

\theoremstyle{definition}

\theoremstyle{remark}
\newtheorem*{remark}{Remark}

\def\c{\mathbb C}
\def\p{\mathbb P}
\def\r{\mathbb R}
\def\z{\mathbb Z}

\def\chp{H_{\c}^2}

\def\bchp{{\partial\chp}}

\def\calC{\mathcal{C}}
\def\calN{\mathcal{N}}
\def\calT{\mathcal{T}}

\def\al{{\alpha}}

\def\Ga{{\Gamma}}

\def\om{{\omega}}

\def\EE{E}

\def\st{\,\,\big|\,\,}
\def\<{\langle}
\def\>{\rangle}

\let\ge=\geqslant
\let\le=\leqslant

\font\msbm=msbm10
\def\semiprod{\hbox{\msbm\char111}}

\def\defit{\it}

\newcommand{\abs}[1]{\left\lvert#1\right\rvert}

\DeclareMathOperator{\Aut}{Aut}
\DeclareMathOperator{\dist}{dist}
\DeclareMathOperator{\Id}{Id}
\DeclareMathOperator{\PU}{PU}

\DeclareMathOperator{\trace}{trace}
\DeclareMathOperator{\U}{U}

\let\Im=\undefined \DeclareMathOperator{\Im}{Im}
\let\Re=\undefined \DeclareMathOperator{\Re}{Re}
\let\mod=\undefined \DeclareMathOperator{\mod}{mod}

\author[Sam Povall]{Sam Povall}
\address{Department of Mathematical Sciences, University of Liverpool, Peach Street, Liverpool L69~7ZL, UK. Current address: Department of Mathematics \& Statistics, University of Melbourne, Parkville, VIC, 3052, Australia}
\email{sam.povall@unimelb.edu.au}
\author[Anna Pratoussevitch]{Anna Pratoussevitch}
\address{Department of Mathematical Sciences, University of Liverpool, Peach Street, Liverpool L69~7ZL, UK}
\email{annap@liverpool.ac.uk}

\title[Complex Hyperbolic Triangle Groups of Type ${[}m,m,0;3,3,2{]}$]{Complex Hyperbolic Triangle Groups\\ of Type $[m,m,0;3,3,2]$}

\begin{date}  {\today} \end{date}

\thanks{
S.P.\ acknowledges the financial support from an EPSRC DTA scholarship at the University of Liverpool
and also the partial support by the International Centre for Theoretical Sciences (ICTS)
during the participation in the programmes Geometry, Groups and Dynamics (ICTS/ggd2017/11) and Surface Group Representations and Geometric Structures (ICTS/SGGS2017/11).
A.P.\ also acknowledges the support from the ICTS}

\subjclass[2010]{Primary 51M10; Secondary 32M15, 22E40, 53C55}

\keywords{complex hyperbolic geometry, triangle groups}

\begin{document}

\maketitle

\begin{abstract}
In this paper we study discreteness of complex hyperbolic triangle groups of type \([m,m,0;3,3,2]\),
i.e.\ groups of isometries of the complex hyperbolic plane
generated by three complex reflections of orders \(3,3,2\) in complex geodesics with pairwise distances \(m,m,0\).
For fixed~$m$, the parameter space of such groups is of real dimension one.
We determine intervals in this parameter space that correspond to discrete and to non-discrete triangle groups.
\end{abstract}

\section{Introduction}

Complex hyperbolic triangle groups are groups of isometries of the complex hyperbolic plane
generated by three complex reflections in complex geodesics.
We will focus on the case of ultra-parallel groups, that is, the case where the complex geodesics are pairwise disjoint.
Unlike real reflections, complex reflections can be of arbitrary order.
If an ultra-parallel complex hyperbolic triangle group is generated by reflections of orders $n_1,n_2,n_3$
in complex geodesics $C_1,C_2,C_3$ with the distance between~$C_{k-1}$ and~$C_{k+1}$ equal to~$m_k$ for~$k=1,2,3$,
then we say that the group is of type $[m_1,m_2,m_3;n_1,n_2,n_3]$.
In this paper, we will study discreteness of ultra-parallel complex hyperbolic triangle groups of type $[m,m,0;3,3,2]$,
i.e.\ two of the reflections are of order~$3$ and one is of order~$2$, the fixed point sets of order~$3$ reflections intersect on the boundary of the complex hyperbolic plane ($m_3=0$)
and the other two distances between fixed point sets coincide ($m_1=m_2$).

The deformation space of groups of type $[m,m,0;3,3,2]$ for a given $m$ is of real dimension one,
a group is determined up to an isometry by the angular invariant $\alpha\in[0,2\pi]$, see section~\ref{sec-background}.
Our main aim is to determine an interval in this one-dimensional deformation space such that for all values of the angular invariant in this interval the corresponding triangle group is discrete.
The main result of the paper is the following proposition:

\begin{proposition}
\label{prop1}
A complex hyperbolic triangle group of type \([m,m,0;3,3,2]\)
with angular invariant \(\alpha\) is discrete if 
\[
  m\ge\log_e(3)
  \quad\mbox{and}\quad
  \cos(\alpha)\le-\frac{1}{2}.
\]
\end{proposition}

\noindent
In the previous works~\cite{WG,Mo,MPP}, the authors considered cases where all three complex reflections are involutions.  
Ultra-parallel triangle groups of types $[m,m,0;2,2,2]$ and $[m,m,2m;2,2,2]$ have been considered in \cite{WG},
while groups of type $[m_1,m_2,0;2,2,2]$ have been considered in \cite{MPP} and \cite{Mo}.

To prove Proposition~\ref{prop1}, we use a version of Klein's combination theorem, adapted to the configurations in question.
Two of the generating reflections share a fixed point on the boundary of the complex hyperbolic plane.
We show that the ultra-parallel triangle group satisfies a compression property by carefully studying the structure of the stabilizer of this fixed point and of its subgroup of Heisenberg translations.
The argument starts in a similar way to that for complex reflections of order~$2$,
however for higher order complex reflections the rank of the group of Heisenberg translations is higher,
leading to a quadratic optimisation problem over~$\z^2$ rather than~$\z$.

On the other hand we obtain the following non-discreteness result using a complex hyperbolic version of Shimizu's lemma:

\begin{proposition}
\label{prop2}
A complex hyperbolic triangle group of type $[m,m,0;3,3,2]$ with angular invariant~$\alpha$ is non-discrete if
\[\cos(\alpha)>1-\frac{1}{12\sqrt{3}\cosh^2\left(\frac{m}{2}\right)}.\]
\end{proposition}

Combining these results, we see that there is a gap between the intervals of discreteness and non-discreteness.
This is illustrated in Figure~\ref{fig-gap}.
The figure shows the $(m,\al)$-space.
The light grey box corresponds to discrete groups (Proposition~\ref{prop1}).
The black area corresponds to non-discrete groups (Proposition~\ref{prop2}).


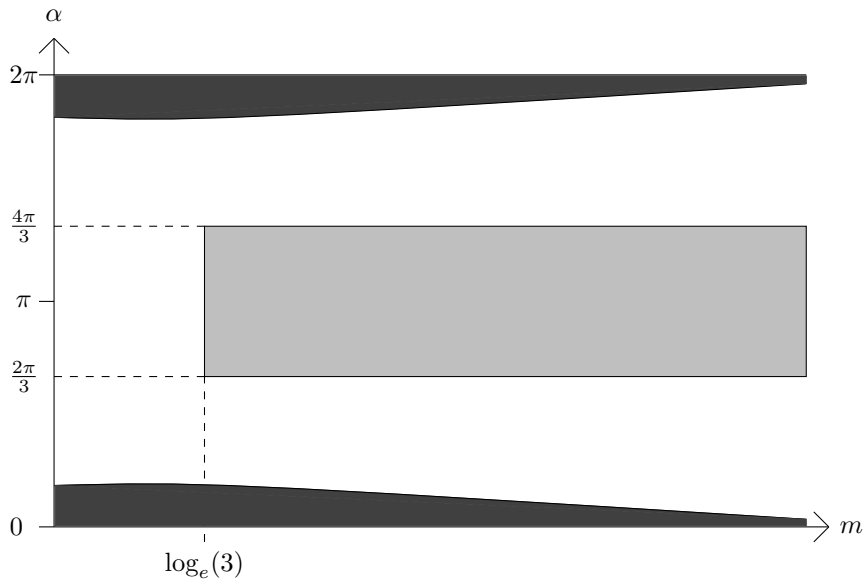
\begin{figure}[h]
\begin{center}
\begin{tikzpicture}
\path[draw] (0.8,1)--(11.3,1);
\path[draw] (0.8,7.014)--(11,7.014);
\path[draw] (1,1)--(1,7.5);
\path[draw] (1,4)--(0.8,4);
\path[draw] (11.1,1.2)--(11.3,1);
\path[draw] (11.1,0.8)--(11.3,1);
\path[draw] (0.8,7.3)--(1,7.5);
\path[draw] (1.2,7.3)--(1,7.5);
\node at (1,7.8) {$\alpha$};
\node at (11.6,1) {$m$};
\node at (0.6,4) {$\pi$};
\node at (0.6,7.014) {$2\pi$};
\node at (0.5,1) {$0$};
\path[draw, dashed] (1,5)--(3,5);
\path[draw, dashed] (1,3)--(3,3);
\node at (0.6,5) {$\frac{4\pi}{3}$};
\node at (0.6,3) {$\frac{2\pi}{3}$};
\path[draw, dashed] (3,0.8)--(3,3);
\node at (3,0.5) {$\log_e(3)$};
\draw[fill= lightgray]  (3,3) -- (11,3) -- (11,5) -- (3,5) -- cycle;
\draw[fill=darkgray,darkgray] (1,6.45) .. controls (3,6.4) .. (11,6.9);
\draw (1,6.449) .. controls (3,6.395) .. (11,6.895);
\draw[fill=darkgray,darkgray]  (1.015,6.45) -- (11,6.9) -- (11,7) -- (1.015,7) -- cycle;
\draw[fill=darkgray,darkgray] (1,1.55) .. controls (3,1.6) .. (11,1.1);
\draw (1,1.551) .. controls (3,1.601) .. (11,1.101);
\draw[fill=darkgray,darkgray]  (1.015,1.55) -- (11,1.1) -- (11,1.015) -- (1.015,1.015) -- cycle;
\end{tikzpicture}
\end{center}
\caption{Discreteness and non-discreteness results in the $(m,\al)$-space.}
\label{fig-gap}
\end{figure}

Ultra-parallel complex hyperbolic triangle groups of type $[m,m,0;n_1,n_2,2]$ with orders $(n_1,n_2)$ other than $(2,2)$ and~$(3,3)$ will be considered in~\cite{Po}.

The paper is organised as follows:
In section~\ref{sec-background} we summarise the necessary background information on complex hyperbolic and Heisenberg geometry.
We introduce the standard parametrisation for ultra-parallel $[m_1,m_2,0;n_1,n_2,n_3]$-triangle groups in section~\ref{sec-param}.
In section~\ref{sec-compression} we use the compression property to derive a discreteness condition for $[m_1,m_2,0;n_1,n_2,n_3]$-groups.
In section~\ref{sec-param-332} we specialise the standard parametrisation to the case of ultra-parallel $[m,m,0;3,3,2]$-triangle groups.
The fixed point sets of order~$3$ reflections intersect on the boundary of the complex hyperbolic plane.
In section~\ref{sec-transl-subgroup} we study the structure of the stabilizer of this intersection point.
In section~\ref{sec-proof-prop1} we use the discreteness conditions from section~\ref{sec-compression} to give a proof of Proposition~\ref{prop1}. 
In section~\ref{sec-proof-prop2} we use a version of Shimizu's lemma to show Proposition~\ref{prop2}. 

We use the following notation: For group elements~$A$ and~$B$, their commutator is $[A,B]=A^{-1}B^{-1}AB$.

\section{Background}

\label{sec-background}

In this section we will give a brief introduction to complex hyperbolic geometry, for further details see \cite{Go,P10}.

\subsection{Complex hyperbolic plane:}
Let $\c^{2,1}$ be the $3$-dimensional complex vector space
equipped with a Hermitian form $\<\cdot,\cdot\>$ of signature $(2,1)$,
e.g.
\[\<z,w\>=z_1\bar{w}_1+z_2\bar{w}_2-z_3\bar{w}_3.\]
If $z\in\c^{2, 1}$ then we know that $\<z, z\>$ is real.
Thus we can define subsets $V_-$, $V_0$ and $V_+$ of $\c^{2,1}$ as follows
\begin{align*}
  V_-&=\{z\in\c^{2,1}\st\<z,z\><0\},\\
  V_0&=\{z\in\c^{2,1}\backslash\{0\}\st\<z,z\>=0\},\\
  V_+&=\{z\in\c^{2,1}\st\<z,z\>>0\}.
\end{align*}
We say that $z\in\c^{2,1}$ is {\defit negative\/}, {\defit null\/} or {\defit positive\/} if $z$ is in $V_-$, $V_0$ or $V_+$ respectively. Define a projection map $\p$ on the points of $\c^{2,1}$ with $z_3\ne0$ as 
\[\p : z=\begin{bmatrix} z_1\\ z_2\\ z_3\end{bmatrix}\mapsto\begin{pmatrix} z_1/z_3\\ z_2/z_3\end{pmatrix}\in\p(\c^{2,1}).\]
That is, provided $z_3\ne0$, 
\[z=(z_1, z_2, z_3)\mapsto[z]=[z_1:z_2:z_3]=\left[\frac{z_1}{z_3}:\frac{z_2}{z_3}:1\right].\]
The {\defit projective model\/} of the complex hyperbolic plane is defined to be the collection of negative lines in $\c^{2,1}$
and its boundary is defined to be the collection of null lines.
That is
\[\chp=\p(V_-)\quad\text{and}\quad\bchp=\p(V_0).\]
The metric on $\chp$, called the {\defit Bergman metric\/}, is given by the distance function~$\rho$ defined by the formula
\[
  \cosh^2\left(\frac{\rho([z],[w])}{2}\right)
  =\frac{\langle{z, w\rangle}\langle{w, z\rangle}}{\langle{z, z\rangle}\langle{w, w\rangle}},
\]
where $[z]$ and $[w]$ are the images of~$z$ and $w$ in $\c^{2,1}$ under the projectivisation map~$\p$. 
The group of holomorphic isometries of~$\chp$ with respect to the Bergman metric
can be identified with the projective unitary group $\PU(2,1)$.

\subsection{Complex geodesics:}
A {\defit complex geodesic\/} is a projectivisation of a 2-di\-men\-sional complex subspace of $\c^{2,1}$.
Any complex geodesic is isometric to \[\{[z:0:1]\st z\in\c\}\] in the projective model.
Any positive vector $c\in V_+$ determines a two-dimensional complex subspace
\[\{z\in\c^{2,1}\st \<c,z\>=0\}.\]
Projecting this subspace we obtain a complex geodesic
\[\p\left(\{z\in\c^{2,1}\st \<c,z\>=0\}\right).\]
Conversely, any complex geodesic is represented by a positive vector $c\in V_+$, 
called a {\defit polar vector\/} of the complex geodesic.
A polar vector is unique up to multiplication by a complex scalar.
We say that the polar vector~$c$ is {\defit normalised\/} if $\<c,c\>=1$.

\bigskip
Let $C_1$ and $C_2$ be complex geodesics with normalised polar vectors~$c_1$ and~$c_2$ respectively.
We call $C_1$ and $C_2$ {\defit ultra-parallel\/} if they have no points of intersection in $\chp\cup\bchp$,
in which case
\[|\<c_1,c_2\>|=\cosh\left(\frac{1}{2}\dist(C_1, C_2)\right)>1,\]
where $\dist(C_1,C_2)$ is the distance between $C_1$ and~$C_2$.
We call $C_1$ and $C_2$ {\defit ideal\/} if they have a point of intersection in $\bchp$,
in which case $|\<c_1,c_2\>|=1$ and $\dist(C_1,C_2)=0$.

\subsection{Complex reflections:}
For a given complex geodesic $C$, a {\defit minimal complex hyperbolic reflection of order~$n$} in~$C$
is the isometry $\iota_C$ in $\PU(2,1)$ of order~$n$ with fixed point set~$C$ given by
\[\iota(z) = -z+(1-\mu)\frac{\<z,c\>}{\<c,c\>}c,\]
where $c$ is a polar vector of~$C$ and $\mu=\exp(2\pi i/n)$.

\subsection{Complex hyperbolic triangle groups:}
A {\defit complex hyperbolic triangle\/} is a triple \((C_1,C_2,C_3)\) of complex geodesics in $\chp$.
A triangle \((C_1,C_2,C_3)\) is a {\defit complex hyperbolic ultra-parallel \([m_1, m_2, m_3]\)-triangle\/}
if the complex geodesics are ultra-parallel at distances $m_k=\dist(C_{k-1},C_{k+1})$ for $k=1,2,3$.
We will allow $m_k=0$ for some or all~$k$.
A {\defit complex hyperbolic ultra-parallel \([m_1,m_2,m_3;n_1,n_2,n_3]\)-triangle group\/}
is a subgroup of \(\PU(2,1)\) generated by complex reflections \(\iota_k\) of order \(n_k\) in the sides \(C_k\)
of a complex hyperbolic ultra-parallel \([m_1,m_2,m_3]\)-triangle \((C_1,C_2,C_3)\).

\subsection{Angular invariant:}
The real dimension of the space of \([m_1,m_2,m_3]\)-triangles
for each fixed triple \(m_1,m_2,m_3\) is equal to one.
We can describe a parametrisation of the space of complex hyperbolic triangles in $\chp$ by means of an angular invariant~$\al$.
We define the {\defit angular invariant\/} $\al$ of the triangle \((C_1, C_2, C_3)\) by
\[\al=\arg\left(\prod_{k=1}^3 \<c_{k-1}, c_{k+1}\>\right),\]
where $c_k$ is the normalised polar vector of the complex geodesic~$C_k$.
We use the following proposition, given in \cite{Pra}, which gives criteria for the existence of a triangle group
in terms of the angular invariant.

\begin{proposition}
\label{traingle-existence}
An $[m_1, m_2, m_3]$-triangle in $\chp$ is determined uniquely up to isometry
by the three distances between the complex geodesics and the angular invariant~$\al$.
For any $\al\in[0, 2\pi]$, an $[m_1, m_2, m_3]$-triangle with angular invariant~$\al$ exists if and only if
\[\cos(\al)<\frac{r_1^2+r_2^2+r_3^2-1}{2r_1r_2r_3},\]
where $r_k=\cosh(m_k/2)$.
\end{proposition}

\noindent
For $m_3=0$ we have $r_3=1$ and the right hand side of the inequality in Proposition~\ref{traingle-existence} is 
\[\frac{r_1^2+r_2^2}{2r_1r_2}\ge1,\]
so the condition on~$\al$ is always satisfied,
i.e.\ for any $\al\in[0, 2\pi]$ there exists an $[m_1, m_2, m_3]$-triangle with angular invariant~$\al$.

\subsection{Heisenberg group:}
\label{heisenberg-group}
The boundary of the complex hyperbolic space can be identified with the {\defit Heisenberg space\/} 
\[\calN=\c\times\r\cup\{\infty\}=\{(\zeta,\nu)\st\zeta\in\c,\nu\in\r\}\cup\{\infty\}.\]
One homeomorphism taking $\bchp$ to $\calN$ is given by the stereographic projection:
\[
  [z_1:z_2:z_3]\mapsto\left(\frac{z_1}{z_2+z_3}, \Im\left(\frac{z_2-z_3}{z_2+z_3}\right)\right)~\text{if}~z_2+z_3\ne0,
  \quad
  [0:z:-z]\mapsto\infty.
\]
The {\defit Heisenberg group\/} is the Heisenberg space~$\calN$ with the group law
\[(\xi_1,\nu_1)*(\xi_2,\nu_2)=(\xi_1+\xi_2,\nu_1+\nu_2+2\Im(\xi_1\bar{\xi_2})).\]
The centre of~$\calN$ consists of elements of the form~$(0,\nu)$ for~$\nu\in\r$.
The Heisenberg group is not abelian but is $2$-step nilpotent.
To see this, observe that 
\[[(\xi_1,\nu_1),(\xi_2,\nu_2)]=(\xi_1,\nu_1)^{-1}*(\xi_2,\nu_2)^{-1}*(\xi_1,\nu_1)*(\xi_2,\nu_2)=(0,4\Im(\xi_1\bar{\xi_2})).\]
Therefore the commutator of any two elements of~$\calN$ lies in the centre.

An alternative description of the Heisenberg group~$\calN$ is as the group of upper triangular matrices
\[\left\{\begin{pmatrix} 1&x&y\\ 0&1&z\\ 0&0&1\end{pmatrix}\st x,y,z\in\r\right\}\]
with the operation of matrix multiplication.
For any integer~$k\ne0$, the subgroup $N_k$ generated by the matrices
\[a=\begin{pmatrix} 1&0&0\\ 0&1&1\\ 0&0&1\end{pmatrix},\quad b=\begin{pmatrix} 1&1&0\\ 0&1&0\\ 0&0&1\end{pmatrix}\quad\text{and}\quad c=\begin{pmatrix} 1&0&\frac{1}{k}\\ 0&1&0\\ 0&0&1\end{pmatrix}\]
is a uniform lattice in~$\calN$ with the presentation
\[N_k=\<a,b,c\st [b,a]=c^k,~[c,a]=[c,b]=1\>.\] 
Moreover, any uniform lattice in~$\calN$ is isomorphic to~$N_k$ for some integer~$k\ne0$, see section~6.1 in~\cite{De}.

\subsection{Chains:}
A complex geodesic in~$\chp$ is homeomorphic to a disc,
its intersection with the boundary of the complex hyperbolic plane is homeomorphic to a circle.
Circles that arise as the boundaries of complex geodesics are called {\defit chains\/}.

\bigskip
There is a bijection between chains and complex geodesics. We can therefore, without loss of generality, talk about reflections in chains instead of reflections in complex geodesics. 

\bigskip
Chains can be represented in the Heisenberg space, for more details see \cite{Go}.
Chains passing through~$\infty$ are represented by vertical straight lines defined by $\zeta = \zeta_0$.
Such chains are called {\defit vertical\/}.
The vertical chain $C_{\zeta_0}$ defined by $\zeta=\zeta_0$ has a polar vector
\[c_{\zeta_0}=\begin{bmatrix}1\\ -\bar{\zeta_0}\\ \bar{\zeta_0}\end{bmatrix}.\]
A chain not containing~$\infty$ is called {\defit finite\/}.
A finite chain is represented by an ellipse whose vertical projection $\c\times\r\rightarrow\c$ is a circle in~$\c$.
The finite chain with centre $(\zeta_0,\nu_0)\in\calN$ and radius $r_0 > 0$ has a polar vector
\[\begin{bmatrix}2\zeta_0 \\ 1+r_0^2-\zeta_0\bar{\zeta_0}+i\nu_0 \\ 1-r_0^2+\zeta_0\bar{\zeta_0}-i\nu_0  \end{bmatrix}\]
and consists of all points~$(\zeta,\nu)\in\calN$ satisfying the equations
\[|\zeta-\zeta_0|=r_0,\quad\nu=\nu_0-2\Im(\zeta\bar{\zeta}_0).\]

\subsection{Heisenberg isometries:}
\label{heisenberg-isometries}
We consider the space~$\calN$ equipped with the {\defit Cygan metric\/},
\[
  \rho_0\left((\zeta_1, \nu_2), (\zeta_2, \nu_2)\right)
  = \Big|\abs{\zeta_1-\zeta_2}^2-i(\nu_1-\nu_2)-2i\Im(\zeta_1\bar{\zeta_2})\Big|^{1/2}.
\]

\bigskip\noindent
A {\defit Heisenberg translation\/}~$T_{(\xi,\nu)}$ by $(\xi,\nu)\in\calN$ is given by
\[(\zeta,\omega)\mapsto(\zeta+\xi,\omega+\nu+2\Im(\xi\bar{\zeta}))=(\xi,\nu)*(\zeta,\omega)\]
and corresponds to the following element in $\PU(2,1)$
\[
  \begin{pmatrix}
    1 & \xi & \xi \\
    -\bar{\xi} & 1-\frac{|\xi|^2-i\nu}{2} & -\frac{|\xi|^2-i\nu}{2}\\ 
    \bar{\xi} & \frac{|\xi|^2-i\nu}{2} & 1+\frac{|\xi|^2-i\nu}{2} 
  \end{pmatrix}.
\]
A special case is a vertical Heisenberg translation~$T_{(0,\nu)}$ by $(0,\nu)\in\calN$ given by
\[(\zeta,\omega)\mapsto(\zeta,\omega+\nu).\]
A {\defit Heisenberg rotation\/}~$R_{\mu}$ by $\mu\in\c$, $|\mu|=1$ is given by
\[(\zeta,\omega)\mapsto(\mu\cdot\zeta,\omega)\]
and corresponds to the following element in $\PU(2,1)$
\[
  \begin{pmatrix}
    \mu & 0 & 0 \\
    0 & 1 & 0\\ 
    0 & 0 & 1
  \end{pmatrix}.
\]
A minimal complex reflection~$\iota_{C_{\varphi}}$ of order~$n$ in a vertical chain~$C_{\varphi}$ with polar vector 
\[c_{\varphi}=\begin{bmatrix}1\\ -\bar{\varphi}\\ \bar{\varphi}\end{bmatrix}\]
 is given by
\[
  (\zeta,\omega)
  \mapsto
  (\mu\zeta+(1-\mu)\varphi,\omega-2|\varphi|^2\Im(1-\mu)+2\Im((1-\mu)\bar{\varphi}\zeta))
\]
and corresponds to the following element in~$\PU(2,1)$
\[
  \begin{pmatrix}
    -\mu & -(1-\mu)\varphi & -(1-\mu)\varphi\\ 
  -(1-\mu)\bar{\varphi} & (1-\mu)|\varphi|^2-1 & (1-\mu)|\varphi|^2\\ 
  (1-\mu)\bar{\varphi} & -(1-\mu)|\varphi|^2 & -(1-\mu)|\varphi|^2-1
  \end{pmatrix},
\] 
where $\mu=\exp(2\pi i/n)$.
The complex reflection $\iota_{C_{\varphi}}$ can be decomposed as a product of a Heisenberg translation and a Heisenberg rotation:
\[\iota_{C_{\varphi}}=R_{\mu}\circ T_{(\xi,\nu)}=T_{(\mu\xi,\nu)}\circ R_{\mu},\]
where 
\[
  \xi=(\bar{\mu}-1)\varphi
  \quad\text{and}\quad
  \nu=-2|\varphi|^2\cdot\Im(1-\mu)=2|\varphi|^2\sin(2\pi/n).
\]
Heisenberg translations, Heisenberg rotations and complex reflections are isometries with respect to the Cygan metric.
The group of all Heisenberg translations is isomorphic to~$\calN$.
The group of all Heisenberg rotations $\{R_{\mu}\st\mu\in\c,~|\mu|=1\}$ is isomorphic to~$\U(1)$.
The group of their products $\calN\semiprod\U(1)$ contains all complex reflections.

\subsection{Products of reflections in chains:}
What effect does the minimal complex reflection of order~$n$ in the vertical chain~$C_\zeta$ have on another vertical chain, $C_\xi$, which intersects $\c\times\{0\}$ at~$\xi$?

\bigskip
We calculate 
\begin{align*}
  \begin{pmatrix}
    -\mu & -(1-\mu)\zeta & -(1-\mu)\zeta \\
    -(1-\mu)\bar{\zeta} & (1-\mu)|\zeta|^2-1 & (1-\mu)|\zeta|^2 \\
    (1-\mu)\bar{\zeta} & -(1-\mu)|\zeta|^2 & -(1-\mu)|\zeta|^2-1
  \end{pmatrix}
  \begin{bmatrix}1\\ -\bar{\xi}\\ \bar{\xi}\end{bmatrix}
  =
  \begin{bmatrix} -\mu\\ -(1-\mu)\bar{\zeta}+\bar{\xi}\\ (1-\mu)\bar{\zeta} -\bar{\xi}\end{bmatrix}.
\end{align*}
This vector is a multiple of
\[
  \begin{bmatrix} 
    1\\ 
    (1-\mu)\bar{\mu}\bar{\zeta}-\bar{\mu}\bar{\xi}\\
    -(1-\mu)\bar{\mu}\bar{\zeta} +\bar{\mu}\bar{\xi}
\end{bmatrix}
=\begin{bmatrix} 1 \\ -\overline{\left(\mu\xi-(\mu-1)\zeta\right)} \\ \overline{\left(\mu\xi-(\mu-1)\zeta\right)} \end{bmatrix}
\]
which is the polar vector of the vertical chain that intersects $\c\times\{0\}$ at $\mu\xi-(\mu-1)\zeta$. 
This corresponds to rotating $\xi$ around $\zeta$ through $\frac{2\pi}{n}$.
So if we have a vertical chain $C_{\xi}$, the minimal complex reflection of order~$n$ in another vertical chain~$C_{\zeta}$ rotates $C_{\xi}$ as a set around $C_{\zeta}$ through $\frac{2\pi}{n}$ (but not point-wise).

\subsection{Bisectors and spinal spheres:}
Unlike in the real hyperbolic space, there are no totally geodesic real hypersurfaces in $\chp$.
An acceptable substitute are the metric bisectors.
Let $z_1, z_2\in\chp$ be two distinct points.
The {\defit bisector equidistant\/} from~$z_1$ and~$z_2$ is defined as
\[\{z\in\chp\st \rho(z_1,z)=\rho(z_2,z)\}.\]
The intersection of a bisector with the boundary of~$\chp$ is a smooth hypersurface in~$\bchp$ called a {\defit spinal sphere\/}, which is diffeomorphic to a sphere. 
An example is the bisector
\[\calC=\{[z:it:1]\in\chp\st |z|^2<1-t^2,~z\in\c,~t\in\r\}.\]
Its boundary, the {\defit unit spinal sphere\/}, can be described as
\[U=\{(\zeta,\nu)\in\calN\st |\zeta|^4+\nu^2=1\}.\]

\section{Parametrisation of complex hyperbolic triangle groups of type $[m_1, m_2, 0;n_1,n_2,n_3]$}

\label{sec-param}

For $r_1, r_2 \ge1$ and $\al\in(0,2\pi)$, let $C_1$, $C_2$ and $C_3$ be the complex geodesics with respective polar vectors
\[
  c_1 = \begin{bmatrix}1 \\ -r_2e^{-i\theta} \\ r_2e^{-i\theta} \end{bmatrix},\quad
  c_2 = \begin{bmatrix}1 \\ r_1e^{i\theta} \\ -r_1e^{i\theta} \end{bmatrix}
  \quad\mbox{and}\quad
  c_3 = \begin{bmatrix}0 \\ 1 \\ 0 \end{bmatrix},
\]
where $\theta=(\pi-\al)/2\in(-\pi/2,\pi/2)$.
The type of triangle formed by $C_1,C_2,C_3$ is an ultra-parallel $[m_1, m_2,0]$-triangle with angular invariant~$\al$,
where $r_k=\cosh(m_k/2)$ for~$k=1,2$.

\bigskip
For $k=1,2,3$, let $\iota_k$ be the minimal complex reflection of order~$n_k$ in the chain~$C_k$.
The group $\<\iota_1,\iota_2,\iota_3\>$ generated by these three complex reflections
is an ultra-parallel complex hyperbolic triangle groups of type $[m_1, m_2, 0;n_1,n_2,n_3]$.
Looking at the arrangement of the chains $C_1$, $C_2$ and $C_3$ in the Heisenberg space $\calN$,
the finite chain~$C_3$ is the (Euclidean) unit circle in $\c\times\{0\}$,
whereas $C_1$ and~$C_2$ are vertical lines
through the points $\varphi_1 = r_2e^{i\theta}$ and $\varphi_2=-r_1e^{-i\theta}$ respectively, see Figure~\ref{fig-chains}.
For~$k=1,2$, the reflection~$\iota_k$ rotates any vertical chain as a set  through $\frac{2\pi}{n_k}$ around~$C_k$.

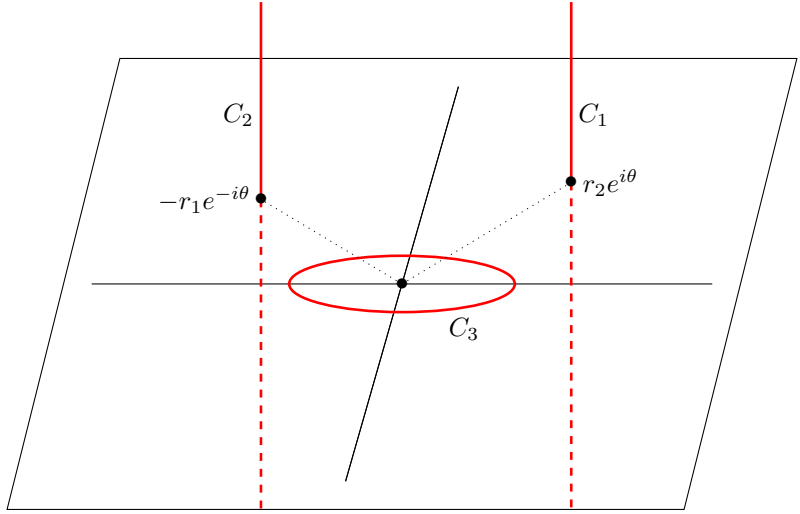
\begin{figure}[h]
\begin{center}
\begin{tikzpicture}[scale=0.75]
  \path[clip] (-8,-5)--(-8,5)--(8,5)--(8,-5)--(-8,-5)--cycle;
  \draw (-7,-4)--(5,-4)--(7,4)--(-5,4)--(-7,-4)--cycle;
  \draw (-5.5,0)--(5.5,0);
  \draw (-1,-3.5)--(1,3.5);
  \draw (-1,-3.5)--(1,3.5);
  \draw[line width=.35mm, red] (-2.5,1.5)--(-2.5,5);
  \draw[line width=.35mm, red] (3,1.8)--(3,5);
  \draw[line width=.35mm, red,dashed] (-2.5,1.5)--(-2.5,-4);
  \draw[line width=.35mm, red,dashed] (3,1.8)--(3,-4);
  \draw[line width=.35mm, red] (0,0) ellipse (2cm and 0.5cm);
  \path[draw,dotted] (0,0)--(-2.5,1.5);
  \path[draw,dotted] (0,0)--(3,1.8);
  \foreach \Point in {(0,0), (-2.5,1.5), (3,1.8)}{\node at \Point {$\bullet$};}
  \node at (1.1,-0.8) {$C_3$};
  \node at (-2.9,3) {$C_2$};
  \node at (3.4,3) {$C_1$};
  \node at (3.7,1.8) {$r_2 e^{i\theta}$};
  \node at (-3.5,1.5) {$-r_1 e^{-i\theta}$};
\end{tikzpicture}
\end{center}
\caption{Chains $C_1$, $C_2$ and $C_3$ (figure from~\cite{MPP}).}
\label{fig-chains}
\end{figure}



\section{Compression property}

\label{sec-compression}

Let $C_1,C_2,C_3$ be chains in~$\calN$ as in the previous section.
Let $\iota_k$ be the minimal complex reflection of order~$n_k$ in the chain~$C_k$ for $k=1,2,3$.
We will assume that $n_3=2$.
To prove the discreteness of the group $\<\iota_1,\iota_2,\iota_3\>$ we will use the following version of Klein's combination theorem discussed in \cite{WG}:

\begin{proposition}

\label{criterion}

If there exist subsets $U_1$, $U_2$ and~$V$ in~$\calN$ with $U_1\cap U_2=\varnothing$ and $V\subsetneq U_1$ such that 
$\iota_3(U_1) = U_2$ and  $g(U_2)\subsetneq V$ for all~$g\ne\Id$ in $\<\iota_1,\iota_2\>$,
then the group $\<\iota_1,\iota_2,\iota_3\>$ is a discrete subgroup of $\PU(2,1)$.
Groups with such properties are called {\defit compressing\/}.
\end{proposition}

\noindent
Projecting the actions of complex reflections~$\iota_1$ and~$\iota_2$ to~$\c\times\{0\}$ we obtain
rotations~$j_1$ and $j_2$ of~$\c$ around $\varphi_1=r_2e^{i\theta}$ and $\varphi_2=-r_1e^{-i\theta}$
through $\frac{2\pi}{n_1}$ and $\frac{2\pi}{n_2}$ respectively.
We will use Proposition~\ref{criterion} to prove the following Lemma:


\begin{lemma}
\label{f(0)}
If $|f(0)|\ge2$ for all $f\ne\Id$ in $\<j_1,j_2\>$ and $|h(0)|\ge2$ for all vertical Heisenberg translations $h\ne\Id$ in $\<\iota_1,\iota_2\>$,
then the group $\<\iota_1,\iota_2,\iota_3\>$ is discrete.
\end{lemma}

\begin{proof}
Consider the unit spinal sphere 
\[U=\{(\zeta, \nu)\in\calN\st|\zeta|^4+\nu^2=1\}.\]
The complex reflection~$\iota_3$ in~$C_3$ is given by
\[\iota_3([z_1:z_2:z_3])=[-z_1:z_2:-z_3]=[z_1:-z_2:z_3].\]
The complex reflection~$\iota_3$ preserves the bisector
\[\calC=\{[z:it:1]\in\chp\st |z|^2<1-t^2, z\in\c, t\in\r\}\]
and hence preserves the unit spinal sphere~$U$ which is the boundary of the bisector~$\calC$.
The complex reflection~$\iota_3$ interchanges the points $[0:1:1]$ and $[0:-1:1]$ in~$\chp$,
which correspond to the points $(0,0)$ and $\infty$ in $\calN$.
Therefore, $\iota_3$ leaves $U$ invariant and switches the inside of $U$ with the outside.

\bigskip
Let $U_1$ be the part of $\calN\backslash U$ outside~$U$, containing~$\infty$,
and let $U_2$ be the part inside~$U$, containing the origin.
Clearly
\[U_1\cap U_2=\varnothing\quad\text{and}\quad \iota_3(U_1)=U_2.\]
Therefore, if we find a subset $V\subsetneq U_1$ such that $g(U_2)\subsetneq V$
for all elements $g\ne\Id$ in $\<\iota_1,\iota_2\>$,
then we will show that $\<\iota_1,\iota_2,\iota_3\>$ is discrete.
Let
\[W=\{(\zeta,\nu)\in\calN\st|\zeta|=1\}\]
be the cylinder consisting of all vertical chains through $\zeta\in\c$ with $|\zeta|=1$.
Let
\[
  W_1=\{(\zeta,\nu)\in\calN\st|\zeta|>1\}
  \quad\text{and}\quad
  W_2=\{(\zeta,\nu)\in\calN\st|\zeta|<1\}
\]
be the parts of $\calN\backslash W$ outside and inside the cylinder~$W$ respectively. 
We have $U_2\subset W_2$ and so $g(U_2)\subset g(W_2)$ for all~$g\in\<\iota_1,\iota_2\>$.
The set $W_2$ is a union of vertical chains.
We know that elements of $\<\iota_1,\iota_2\>$ map vertical chains to vertical chains.
There is also a vertical translation on the chain itself.
Therefore, we look at both the intersection of the images of~$W_2$ with $\c\times\{0\}$ and the vertical displacement of~$W_2$. 

\bigskip
Elements of $\<\iota_1,\iota_2\>$ move the intersection of $W_2$ with $\c\times\{0\}$ by rotations $j_1$ and $j_2$
around \(r_2e^{i\theta}\) and \(-r_1e^{-i\theta}\) through $\frac{2\pi}{n_1}$ and $\frac{2\pi}{n_2}$ respectively.
Provided that the interior of the unit circle is mapped completely off itself under all non-identity elements in $\<j_1,j_2\>$,
then the same is true for $W_2$ and hence for $U_2$ under all elements in $\<\iota_1,\iota_2\>$ that are not vertical Heisenberg translations.

\bigskip
A vertical Heisenberg translation will shift~$W_2$ and its images $g(W_2)$ vertically by the same distance,
hence the same is true for $U_2$ and its images $g(U_2)$.

\bigskip
We choose $V$ to be the union of all the images of $U_2$ under all non-vertical elements of $\<\iota_1,\iota_2\>$.
This subset will satisfy the compressing conditions
assuming that the interior of the unit circle is mapped off itself by any non-identity element in $\<j_1,j_2\>$
and that the interior of the unit spinal sphere $U$ is mapped off itself
by any non-identity vertical Heisenberg translation in $\<\iota_1,\iota_2\>$.
Since the radius of the unit circle is preserved under rotations,
we need to show that the origin is moved the distance of at least twice the radius of the circle:
\[|f(0)|\ge2\quad\text{for all}~f\in\<j_1,j_2\>,~f\ne\Id.\]
Since vertical translations shift the spinal spheres vertically,
we need to show that they shift by at least the height of the spinal sphere:
\[|h(0)|\ge2\quad\text{for all vertical Heisenberg translations}~h\in\<\iota_1,\iota_2\>,~h\ne\Id.\]
We see that the conditions of this Lemma ensure that the sets $U_1$, $U_2$ and~$V$ satisfy the conditions of Proposition~\ref{criterion}.
\end{proof}

\section{Parametrisation of complex hyperbolic triangle groups of type $[m,m,0;3,3,2]$}

\label{sec-param-332}

\noindent
We will now focus on the case of $[m_1,m_2,0;n_1,n_2,n_3]$-groups with
\[m_1=m_2=m,\quad n_1=n_2=3\quad\text{and}\quad n_3=2.\]
In this case the setting described in section~\ref{sec-param} is as follows.
We consider the following configuration of chains in~$\calN$:
$C_3$ is the (Euclidean) unit circle in $\c\times\{0\}$, whereas $C_1$ and~$C_2$ are vertical lines
through the points $\varphi_1 = re^{i\theta}$ and $\varphi_2=-re^{-i\theta}$ respectively,
where $r=\cosh(m/2)$ and $\theta\in(-\pi/2,\pi/2)$.
The type of triangle formed by $C_1,C_2,C_3$ is an ultra-parallel $[m,m,0]$-triangle with angular invariant $\al=\pi-2\theta\in(0,2\pi)$.
We will consider the ultra-parallel triangle group $\Ga=\<\iota_1,\iota_2,\iota_3\>$
generated by the minimal complex reflections $\iota_1,\iota_2,\iota_3$ of orders~$3,3,2$ in the chains~$C_1,C_2,C_3$ respectively.

The description of complex reflections in section~\ref{heisenberg-isometries} in this case is as follows:
The reflection $\iota_k$ for $k=1,2$ is given by 
\[
  (\zeta,\omega)
  \mapsto
  (\mu\zeta+(1-\mu)\varphi_k,\omega+2|\varphi_k|^2\Im(1-\mu)+2\Im((1-\mu)\bar{\varphi_k}\zeta)),
\]
where $\mu=\exp(2\pi i/3)$,
and can be decomposed into a product of a Heisenberg translation and a Heisenberg rotation:
\[\iota_k=R_{\mu}\circ T_{(\xi_k,\nu_k)}=T_{(\mu\xi_k,\nu_k)}\circ R_{\mu},\]
where 
\[
  \xi_k=(\bar{\mu}-1)\varphi_k
  \quad\text{and}\quad
  \nu_k=-2|\varphi_k|^2\cdot\Im(1-\mu)=2|\varphi_k|^2\sin(2\pi/3).
\]
For $k=1,2$, the reflection~$\iota_k$ rotates any vertical chain as a set  through $\frac{2\pi}{3}$ around~$C_k$.

\section{Subgroup of Heisenberg translations}

\label{sec-transl-subgroup}

\noindent
Let $\Ga=\<\iota_1,\iota_2,\iota_3\>$ be as in section~\ref{sec-param-332}.
In this section we will consider the structure of the subgroup~$\EE=\<\iota_1,\iota_2\>$ in more detail.

\begin{proposition}
\label{transl-subgroup}
Let $\calT$ be the subgroup of all Heisenberg translations in $\EE$.
Every element of~$\EE$ can be written as a product of a Heisenberg translation and a power of~$\iota_1$.
The group~$\calT$ is generated by the elements
\[T_1=\iota_2\iota_1\iota_2\quad\text{and}\quad T_2=\iota_1\iota_1\iota_2.\]
Let $H=[T_1,T_2]=(\iota_1\iota_2)^3$.
Every element of~$\calT$ is of the form $T_1^xT_2^yH^n$ for some $x,y,n\in\z$.
The elements $T_1$, $T_2$, $H$ are Heisenberg translations by
\begin{align*}
  (v_1,t_1)&=(2r\sqrt{3}\cos(\theta)\cdot i,12\sqrt{3}r^2\cos^2(\theta)),\\
  (v_2,t_2)&=(r\cos(\theta)\cdot(3+i\sqrt{3}),12r^2\sin(\theta)\cos(\theta)),\\
  (0,\nu)&=(0,24r^2\sqrt{3}\cos^2(\theta))
\end{align*}
respectively.
The subgroup of vertical Heisenberg translations in $\EE$ is an infinite cyclic group generated by~$H$.
The shortest non-trivial vertical translations in $\EE$ are $H$ and~$H^{-1}$.
\end{proposition}

\begin{proof}
We can write every element in $\EE$ as a word in the generators~$\iota_1^{\pm1}$ and~$\iota_2^{\pm1}$.
Using the relations $\iota_1^{-1}=\iota_1^2$ and $\iota_2^{-1}=\iota_2^2$ we can rewrite it as a word in just~$\iota_1$ and~$\iota_2$.
Consider the words $\iota_{k_1k_2k_3}=\iota_{k_1}\iota_{k_2}\iota_{k_3}$ of length~$3$.
Using the decomposition $\iota_k=R_{\mu}\circ T_{(\xi_k,\nu_k)}=T_{(\mu\xi_k,\nu_k)}\circ R_{\mu}$ (section~\ref{sec-param-332}),
we can write
\begin{align*}
  \iota_{k_1k_2k_3}
  &=(R_{\mu}\circ T_{(\xi_{k_1},\nu_{k_1})})\circ(R_{\mu}\circ T_{(\xi_{k_2},\nu_{k_2})})\circ(R_{\mu}\circ T_{(\xi_{k_3},\nu_{k_3})})\\
  &=(R_{\mu})^3\circ T_{(\mu\xi_{k_1},\nu_{k_1})}\circ T_{(\mu^2\xi_{k_2},\nu_{k_2})}\circ T_{(\xi_{k_3},\nu_{k_3})}\\
  &=T_{(\mu\xi_{k_1},\nu_{k_1})}\circ T_{(\mu^2\xi_{k_2},\nu_{k_2})}\circ T_{(\xi_{k_3},\nu_{k_3})},
\end{align*}
hence $\iota_{k_1k_2k_3}$ is a Heisenberg translation.
Let $f\in\EE=\<\iota_1,\iota_2\>$.
We can write~$f$ as a product of some words of length~$3$ and one word of length at most~$2$.
Moreover, using the relations $\iota_2=\iota_{211}\cdot\iota_1$, $\iota_2^2=\iota_{221}\cdot\iota_1^2$, $\iota_1\iota_2=\iota_{121}\cdot\iota_1^2$ and $\iota_2\iota_1=\iota_{211}\cdot\iota_1^2$,
we can rewrite~$f$ as a product of some words of length~$3$ and a power of~$\iota_1$.
Using the relations $\iota_1^3=\iota_2^3=\Id$ we see that all words $\iota_{k_1k_2k_3}$ of length~$3$ can be expressed in terms of $T_1=\iota_{212}$ and $T_2=\iota_{112}$
as $\iota_{221}=T_2^{-1}$, $\iota_{122}=T_2T_1^{-1}$, $\iota_{211}=T_1T_2^{-1}$ and $\iota_{121}=T_2T_1^{-1}T_2^{-1}$.
Hence $f$ can be written as a product of an element in~$\<T_1,T_2\>$ and an element $w\in\<\iota_1\>$,
and $f$ is a Heisenberg translations if and only if~$w=\Id$.
Therefore $\calT=\<T_1,T_2\>$.

Let $H=[T_1,T_2]\in\calT$.
As a commutator of two Heisenberg translations, the element~$H$ is a vertical Heisenberg translation and lies in the centre of~$\calN$, hence $[H,T_1]=[H,T_2]=1$.
Direct computation shows that $T_2HT_2^{-1}=(\iota_1\iota_2)^3$.
On the other hand, $T_2H=HT_2$ implies $T_2HT_2^{-1}=HT_2T_2^{-1}=H$, hence $H=(\iota_1\iota_2)^3$.
Using the relations  $HT_1=T_1H$, $HT_2=T_2H$ and $T_2T_1=T_1T_2H^{-1}$,
every element of~$\calT$ can be written in the form $T_1^xT_2^yH^n$ for some $x,y,n\in\z$.
The elements~$T_1$ and~$T_2$ are Heisenberg translation by $(v_1,t_1)$ and~$(v_2,t_2)$ respectively.
The commutator $H=[T_1,T_2]$ is a vertical Heisenberg translation by $\nu=4\Im(v_1\bar{v}_2)$.
We determine $(v_k,t_k)$ and~$\nu$ by direct computation.
Projection to~$\c$ maps $H$ to the identity, $T_k$ to the Euclidean translation by~$v_k$ and $T_1^xT_2^yH^n$ to the Euclidean translation by $x v_1+y v_2$.
Hence $T_1^xT_2^yH^n$ is a vertical translation if and only if $x=y=0$, i.e.\ if it is a power of~$H$.
Therefore the subgroup of vertical Heisenberg translations in $\EE$ is generated by~$H$.
\end{proof}

\begin{remark}
The group~$\calT$ has the presentation
\[\calT=\<T_1,T_2,H\st [T_1,T_2]=H,~[H,T_1]=[H,T_2]=1\>\]
and is isomorphic to the uniform lattice $N_1$ as defined in section~\ref{heisenberg-group}.
\end{remark}

\begin{remark}
An alternative approach to the understanding of the structure of the subgroup~$\EE=\<\iota_1,\iota_2\>$
is to use the classification of almost-crystallographic groups by Dekimpe~\cite{De}.
An {\defit almost-crystallographic group\/} is a uniform discrete subgroup~$E$ of $G\semiprod C$,
where $G$ is a connected, simply connected nilpotent Lie group and $C$ is a maximal compact subgroup of~$\Aut(G)$.
As a discrete subgroup of $\calN\semiprod\U(1)$ (see section~\ref{heisenberg-isometries}),
the group~$\EE$ is an almost-crystallographic group with $G=\calN$ and $\U(1)\subset C\subset\Aut(\calN)$.
The projection of~$\EE=\<\iota_1,\iota_2\>$ to~$\c$ is a wallpaper group~$Q=\<j_1,j_2\>$,
where $j_k$ is the rotation of $\c$ around $\varphi_k$ through~$2\pi/3$ obtained by projecting~$\iota_k$ to~$\c$.
The wallpaper group $Q=\<j_1,j_2\>$ is generated by two order~$3$ rotations and has a presentation
\[\<j_1,j_2\st j_1^3=j_2^3=(j_1j_2)^3=1\>.\]
The standard notation for this wallpaper group is~{\bf p3}, see for example~\cite{BB}.
In the classification of $3$-dimensional almost-crystallographic groups in section~7.1 of~\cite{De}, the wallpaper group~{\bf p3} appears in case~13 on page~164.
In this case the group~$E$ is generated by elements $a,b,c,\al$ with relations
\[[b,a]=c^{k_1},~[c,a]=[c,b]=[c,\al]=1,~\al a=b\al c^{k_2},~\al b=a^{-1}b^{-1}\al c^{k_3},~\al^3=c^{k_4}.\]
We consider the generators $\iota_1=\al$ and $\iota_2=\al a$ so that $\al^3=(\al a)^3=1$.
The hypothesis $\al^3=1$ implies $k_4=0$.
The hypothesis $(\al a)^3=1$ can be rewritten as 
\begin{align*}
  1&=(\al a)^3=\al a\al(a\al)a
  =\al a\al(b^{-1}\al)b^{-1}ac^{k_3}
  =\al a(\al\al)a^{-1}b^{-1}ac^{k_2+k_3}\\
  &=(b a^{-1}b^{-1}a)c^{2k_2+k_3}
  =b([b,a])^{-1}b^{-1}c^{2k_2+k_3}
  =c^{-k_1+2k_2+k_3},
\end{align*}
hence $-k_1+2k_2+k_3=0$.
The translations~$T_1$ and~$T_2$ in Proposition~\ref{transl-subgroup} are
\[
  T_1=\iota_2\iota_1\iota_2=(\al a)\al^2 a=ba c^{k_2}
  \quad\text{and}\quad
  T_2=\iota_1\iota_1\iota_2=\al^3 a=a.
\]
Their commutator is
\[
  H
  =[T_1,T_2]
  =(ba c^{k_2})^{-1}a^{-1}(ba c^{k_2})a
  =a^{-1}b^{-1}a^{-1}baa
  =a^{-1}[b,a]a
  =c^{k_1}.
\]
On the other hand, the kernel of the map $\EE=\<\iota_1,\iota_2\>\to\<j_1,j_2\>$ given by $\iota_1\mapsto j_1$, $\iota_2\mapsto j_2$ is generated by $(\iota_1\iota_2)^3$.
We calculate
\[
  (\iota_1\iota_2)^3=(\al^2a)^3=\al^2 a\al^2(a\al^2)a
  =\al^2 a(\al b)a c^{k_2}
  =\al^2 b^{-1}(\al a)c^{k_2+k_3}  
  =c^{2k_2+k_3}.
\]
Using $-k_1+2k_2+k_3=0$ we can rewrite this as $(\iota_1\iota_2)^3=c^{k_1}$.
Hence the element $H=[T_1,T_2]=(\iota_1\iota_2)^3=c^{k_1}$ is the shortest vertical Heisenberg translation in $\EE=\<\iota_1,\iota_2\>$.
\end{remark}

\section{Proof of Proposition~\ref{prop1}}

\label{sec-proof-prop1}

\noindent
Let $\Ga=\<\iota_1,\iota_2,\iota_3\>$ be as in section~\ref{sec-param-332}.
In this section we will use Lemma~\ref{f(0)} to find conditions for the group $\Ga$ to be discrete.

\begin{proof}
We need to check that the conditions of Lemma~\ref{f(0)} are satisfied.
Note that $m\ge\log_e(3)$ implies
\[r=\cosh\left(\frac{m}{2}\right)\ge\frac{2}{\sqrt{3}}.\]
We first check that $|h(0)|\ge2$ for all vertical Heisenberg translations $h\ne\Id$ in $\<\iota_1,\iota_2\>$.
Any vertical translation in $\<\iota_1,\iota_2\>$ is a power of the vertical translation~$H$ by $(0,24r^2\sqrt{3}\cos^2(\theta))$.
We need the displacement of each vertical translation~$H^n$, $n\ne0$, to be at least the height of the spinal sphere, i.e.
\[24r^2\sqrt{3}\cos^2(\theta)\ge2\Longleftrightarrow r^2\cos^2(\theta)\ge\frac{\sqrt{3}}{36}.\]
The hypothesis $\cos(\al)\le-\frac{1}{2}$ for $\al\in(0,2\pi)$ implies $\frac{2\pi}{3}\le\al\le\frac{4\pi}{3}$ and hence $|\theta|=\left|\frac{\pi-\al}{2}\right|\le\frac{\pi}{6}$.
For $\cos(\theta)\ge\frac{\sqrt{3}}{2}$ and $r\ge\frac{2}{\sqrt{3}}$ we have
\[
  r^2\cos^2(\theta)\ge
  1>\frac{\sqrt{3}}{36},
\]
hence the condition $|h(0)|\ge2$ is satisfied for all vertical translations $h\ne\Id$ in $\<\iota_1,\iota_2\>$.

\bigskip
We will now check that $|f(0)|\ge2$ for all $f\ne\Id$ in $\<j_1,j_2\>$.
We can write every element~$f$ in $\<j_1,j_2\>$ as a word in the generators~$j_1$ and~$j_2$.
Figure~\ref{fig2} shows the points $f(0)$ for all words~$f$ of length up to~$6$ in the case $r=1$ and $\theta=0$.

\begin{figure}[h]
\centering
\includegraphics[width=0.5\textwidth]{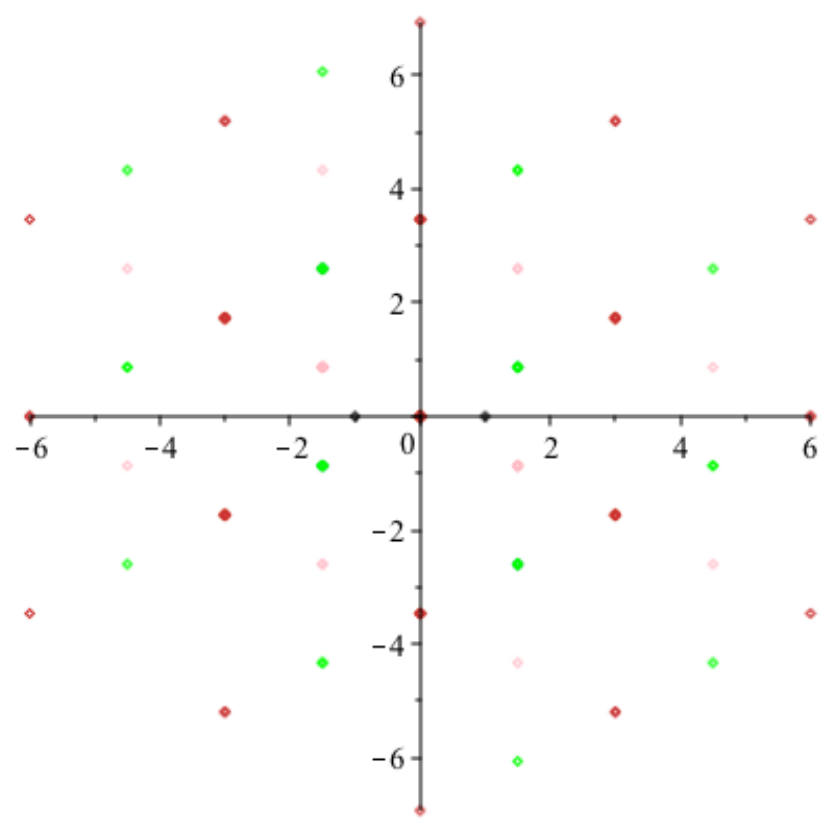}
\caption{Points $f(0)$ for all words $f$ up to length $6$.}
\label{fig2}
\end{figure}

The group $\<j_1,j_2\>$ is the projection to~$\c$ of the group~$\EE=\<\iota_1,\iota_2\>$.
For $k=1,2$, projecting $\iota_k$ to~$\c$, we obtain a rotation~$j_k$ of $\c$ through~$\frac{2\pi}{3}$ around $\varphi_k$.
These rotations are given by $j_k(z)=\mu\cdot z+(1-\mu)\cdot\varphi_k$, where $\mu=\exp(2\pi i/3)$.
According to Proposition~\ref{transl-subgroup}, every element of~$\EE$ is of the form $T_1^xT_2^yH^n\iota_1^\ell$ for some $x,y,n\in\z$ and $\ell\in\{0,1,2\}$,
where $T_1=\iota_2\iota_1\iota_2$, $T_2=\iota_1\iota_1\iota_2$ and $H=[T_1,T_2]$ are Heisenberg translations by $(v_1,t_1)$, $(v_2,t_2)$ and~$(0,\nu)$ respectively and
\[v_1=2r\sqrt{3}\cos(\theta)\cdot i,\quad v_2=r\cos(\theta)\cdot(3+i\sqrt{3}).\]
Projection to~$\c$ maps $H$ to the identity, $T_k$ to the Euclidean translation by~$v_k$, $T_1^xT_2^yH^n$ to the Euclidean translation by $x v_1+y v_2$ and $\iota_1$ to the rotation~$j_1$,
therefore every element of~$\<j_1,j_2\>$ is a product of a translation by~$xv_1+yv_2$ for some~$x,y\in\z$ and a rotation $j_1^\ell$ for some~$\ell\in\{0,1,2\}$.
Hence every point in the orbit of~$0$ under $\<j_1,j_2\>$ is of the form $p+xv_1+yv_2$, where $x,y\in\z$ and
\[p\in\{0,j_1(0),j_1^2(0)\}=\{0,(1-\mu)\varphi_1,(1-\bar\mu)\varphi_1\}.\]
Using $|v_1|^2=|v_2|^2=2\Re(v_1\bar{v}_2)=12r^2\cos^2(\theta)$,
we calculate
\begin{align*}
  &|p+xv_1+yv_2|^2\\
  &=x^2|v_1|^2+y^2|v_2|^2+2xy\Re(v_1\bar{v}_2)+2x\Re(p\bar{v}_1)+2y\Re(p\bar{v}_2)+|p|^2\\
  &=12r^2\cos^2(\theta)\cdot(x^2+xy+y^2)+2x\Re(p\bar{v}_1)+2y\Re(p\bar{v}_2)+|p|^2.
\end{align*}
We make a coordinate change $u=y-x$ and $v=x+y$, that is $x=(v-u)/2$ and $y=(u+v)/2$.
Points $(x,y)\in\z^2$ are mapped to points $(u,v)\in\z^2$ with $u\equiv v\mod 2$.
We obtain
\begin{align*}
  |p+xv_1+yv_2|^2
  &=3r^2\cos^2(\theta)\cdot(u^2+3v^2-2au-6bv+a^2+3b^2)\\
  &=3r^2\cos^2(\theta)\cdot((u-a)^2+3(v-b)^2),\\
\end{align*}
where
\begin{align*}
  a
  &=\frac{\Re(p(\bar{v}_1-\bar{v}_2))}{6r^2\cos^2(\theta)}
  =-\frac{\Re\left(p(3+i\sqrt{3})\right)}{6r\cos(\theta)},\\
  b&=-\frac{\Re(p(\bar{v}_1+\bar{v}_2))}{18r^2\cos^2(\theta)}
  =-\frac{\Re\left(p(1-i\sqrt{3})\right)}{6r\cos(\theta)}
\end{align*}
and
\[a^2+3b^2=\frac{|p|^2}{3r^2\cos^2(\theta)}.\]
Our aim is to show that $|p+xv_1+yv_2|^2\ge3r^2$ for all~$(x,y)\in\z^2$
excluding the case $p=0$, $x=y=0$ that corresponds to $f=\Id$. 
This is equivalent to $(u-a)^2+3(v-b)^2\ge\sec^2(\theta)$ for all~$(u,v)\in\z^2$ with $u\equiv v\mod 2$
excluding the case $a=b=u=v=0$.
Note that this inequality is always satisfied if $|u-a|\ge\sec(\theta)$ or $|v-b|\ge\sec(\theta)/\sqrt{3}$,
so we only need to check that
\[g(u,v)=(u-a)^2+3(v-b)^2-\sec^2(\theta)\ge0\]
for all $(u,v)\in\z^2$ with $u\equiv v\mod 2$ inside the bounding box
\[\left(a-\sec(\theta),a+\sec(\theta)\right)\times\left(b-\frac{\sec(\theta)}{\sqrt{3}},b+\frac{\sec(\theta)}{\sqrt{3}}\right).\]
In the following table we list the values of $a$, $b$ and $a^2+3b^2$
in terms of $t=\tan(\theta)$ and $\mu=\exp(2\pi i/3)=-\frac{1-i\sqrt{3}}{2}$ for $w\in\{\Id,j_1,j_1^2\}$:
\begin{center}
\begin{tabular}{|c|c|c|c|c| }
\hline\(w\) & \(p=w(0)\) & \(a\) & \(b\) & \(a^2+3b^2\) \\
\hline\(\Id\) & \(0\) & \(0\) & \(0\) & \(0\) \\
\hline\(j_1\) & $(1-\mu)\cdot\varphi_1$ & \(-1\) & \(-\frac{t}{\sqrt{3}}\) & \(t^2+1\) \\
\hline\(j_1^2\) & $(1-\bar\mu)\cdot\varphi_1$ & \(\frac{1}{2}(t\sqrt{3}-1)\) & \(-\frac{1}{6}(3+t\sqrt{3})\) & \(t^2+1\) \\
\hline
\end{tabular}
\end{center}


\bigskip\noindent
Under the assumption $|\theta|\le\frac{\pi}{6}$
we have $t=\tan(\theta)\in[-d,d]$ and $\sec(\theta)\in[1,2d]$, where $d=1/\sqrt{3}\approx0{.}577$.
In each of the three cases we list the bounds on $a$ and~$b$ and the size of the bounding box
\[(\min(a)-2d,\max(a)+2d)\times(\min(b)-2/3,\max(b)+2/3).\]
We then calculate
\begin{align*}
  g(u,v)
  &=(u-a)^2+3(v-b)^2-\sec^2(\theta)\\
  &=u^2+3v^2-2au-6bv+(a^2+3b^2)-(t^2+1)
\end{align*}
and check that $g(u,v)\ge0$ for all $(u,v)\in\z^2$ with $u=v\mod2$ inside the bounding box.
\begin{enumerate}[$\bullet$]
\item
$w=\Id$, $a=b=0$:
The bounding box
\[(-2d,2d)\times(-2/3,2/3)\subset(-2,2)\times(-1,1)\]
contains only one point~$(u,v)\in\z^2$ with $u=v\mod2$, the point $(u,v)=(0,0)$,
which corresponds to the excluded case $f=\Id$.
\item
$w=j_1$, $a=-1$, $b=-t/\sqrt{3}\in[-1/3,1/3]$:
The bounding box
\[(-1-2d,-1+2d)\times(-1,1)\subset(-3,1)\times(-1,1)\]
contains points $(0,0)$ and $(-2,0)$.
The function
\begin{align*}
  g(u,v)
  &=u^2+3v^2+2u+2tv\sqrt{3}+(t^2+1)-(t^2+1)\\
  &=u^2+3v^2+2u+2tv\sqrt{3}
\end{align*}
is non-negative: $g(0,0)=g(-2,0)=0$.
\item
$w=j_1^2$, $a=\frac{1}{2}(t\sqrt{3}-1)\in[-1,0]$, $b=-\frac{1}{6}(3+t\sqrt{3})\in[-2/3,-1/3]$:
The bounding box
\[(-1-2d,2d)\times(-4/3,1/3)\subset(-3,2)\times(-2,1)\]
contains points $(1,-1)$, $(0,0)$, $(-1,-1)$ and $(-2,0)$.
The function
\begin{align*}
  g(u,v)
  &=u^2+3v^2-u(t\sqrt{3}-1)+v(3+t\sqrt{3})+(t^2+1)-(t^2+1)\\
  &=u^2+3v^2+u+3v-(u-v)t\sqrt{3}
\end{align*}
 is non-negative:
\begin{align*}
 &g(0,0)=g(-1,-1)=0,~
 g(1,-1)=2-2t\sqrt{3}\ge0,~
 g(-2,0)=2+2t\sqrt{3}\ge0.
\end{align*}
\end{enumerate}
In all cases we have shown that $g(u,v)\ge0$, hence $|p+xv_1+yv_2|^2\ge3r^2$.
Under the assumption $m\ge\log_e(3)$ we have $3r^2=3\cosh^2\left(\frac{m}{2}\right)\ge4$.
Therefore $|f(0)|\ge2$ for all $f\ne\Id$ in $\<j_1,j_2\>$.
Hence all conditions of Lemma~\ref{f(0)} are satisfied
and we can conclude that the group $\<\iota_1,\iota_2,\iota_3\>$ is discrete.
\end{proof}

\section{Proof of Proposition~\ref{prop2}}

\label{sec-proof-prop2}

\noindent
Let $\Ga=\<\iota_1,\iota_2,\iota_3\>$ be an ultra-parallel $[m,m,0;3,3,2]$-triangle group as in section~\ref{sec-param-332}.
In this section we will use the following complex hyperbolic version of Shimizu's Lemma introduced in~\cite{P92,P94,P97} to find conditions for the group $\Ga$ not to be discrete.

\begin{lemma}
\label{shimizu}
Let $\Ga$ be a discrete subgroup of $\PU(2,1)$.
Let $g\in\Ga$ be a Heisenberg translation by $(\xi, \nu)$ and $h=(h_{ij})_{1\le i,j\le 3}\in\Ga$
be an element with $h(\infty)\ne\infty$,
then 
\[r_h^2 \le\rho_0(g(h^{-1}(\infty)),h^{-1}(\infty))\rho_0(g(h(\infty)), h(\infty)) +4\abs{\xi}^2,\]
where $\rho_0$ is the Cygan metric on~$\calN$
and 
\[r_h = \sqrt{\frac{2}{\abs{h_{22}-h_{23}+h_{32}-h_{33}}}}\]
is the radius of the isometric sphere of~$h$.
\end{lemma}

\bigskip\noindent
We will now prove Proposition~\ref{prop2}:


\begin{proof}
We will apply Lemma~\ref{shimizu} to the vertical Heisenberg translation $g=(\iota_1\iota_2)^3$
and the element $h=\iota_3$ in $\Ga=\<\iota_1,\iota_2,\iota_3\>$.
The matrix of the element $h=\iota_3=\iota_3^{-1}$ is
\[h=h^{-1}=\begin{pmatrix} -1 & 0 & 0 \\ 0 & 1 & 0 \\ 0 & 0 & -1 \end{pmatrix}.\]
The radius of the isometric sphere of~$h$ is $r_h=1$.
To calculate $h(\infty)$ we first map $\infty$ from the Heisenberg space to the boundary of complex hyperbolic 2-space.
That is,
\[\infty\mapsto [0:1:-1]\in\bchp.\]
We apply $h$ to this point,
\[h([0:1:-1])= [0:1:1]\in\bchp.\]
Note that $h(\infty)\ne\infty$.
Mapping this point back to the Heisenberg space,
\[[0:1:1]\mapsto(0,0)\in\calN.\]
For a vertical Heisenberg translation~$g$, we have $\xi=0$ and $\rho_0(g(\zeta,\om),(\zeta,\om))=|\nu|^{\frac{1}{2}}$ for all~$(\zeta,\om)\in\calN$.
Substituting these values into the inequality given in Lemma~\ref{shimizu}, we obtain that if $|\nu|<1$ then the group is not discrete.
From Proposition~\ref{transl-subgroup} we know that $g=(\iota_1\iota_2)^3$ is a vertical Heisenberg translation by $(0,\nu)$ with $\nu=24\sqrt{3}r^2\cos^2(\theta)$,
hence the group~$\Ga$ is not discrete if
\[\cos^2(\theta)<\frac{1}{24\sqrt{3}r^2}.\]
Using $\cos(\al)=1-2\cos^2(\theta)$,
we conclude that the group~$\Ga$ is not discrete provided that
\[\cos(\al)>1-\frac{1}{12\sqrt{3}r^2}=1-\frac{1}{12\sqrt{3}\cosh^2\left(\frac{m}{2}\right)}.\qedhere\]
\end{proof}

\bigskip\noindent
{\bf Acknowledgements:}
We would like to thank John Parker for many helpful suggestions, in particular for pointing out a gap in an earlier version of Lemma~\ref{f(0)} and for telling us about the work of Karel Dekimpe.
We would like to thank the referees for their valuable comments.


\def\cprime{$'$}
\providecommand{\bysame}{\leavevmode\hbox to3em{\hrulefill}\thinspace}
\providecommand{\MR}{\relax\ifhmode\unskip\space\fi MR }
\providecommand{\MRhref}[2]{%
  \href{http://www.ams.org/mathscinet-getitem?mr=#1}{#2}
}
\providecommand{\href}[2]{#2}

\end{document}